\documentclass{llncs}
\usepackage[margin=1.4in]{geometry}
\usepackage[utf8]{inputenc}
\usepackage{amsmath, amssymb, amsfonts}

\usepackage{bbm}
\usepackage{latexsym}
\usepackage{color}
\usepackage{soul}
\usepackage{cite}
\usepackage{xcolor}
\usepackage[colorlinks=false]{hyperref} 
\hypersetup{pdfstartview=FitH,pdftoolbar=false}
\usepackage{microtype}
\usepackage{cleveref}

\usepackage{caption}
\usepackage{subcaption}
\usepackage{tikz}
\usetikzlibrary{arrows,calc}

\usepackage{todonotes}

\newcommand{\R}{\mathbb{R}}

\newcommand{\N}{\mathbb{N}}

\newcommand{\conv}{\mathop\mathrm{conv}\nolimits}
\newcommand{\aff}{\mathop\mathrm{aff}\nolimits}

\newcommand{\rec}{\mathop\mathrm{rec}\nolimits}
\renewcommand{\subset}{\subseteq}
\renewcommand{\supset}{\supseteq}
\newcommand{\unit}[1]{\mathbf{e}_{#1}}
\newcommand{\zero}{\mathbf{0}}
\newcommand{\C}{\mathcal{C}}
\newcommand{\B}{\mathcal{B}}
\newcommand{\V}{\mathcal{V}}

\crefname{theorem}{Theorem}{Theorems}
\crefname{lemma}{Lemma}{Lemmas}
\crefname{corollary}{Corollary}{Corollaries}
\crefname{proposition}{Proposition}{Propositions}
\crefname{claim}{Claim}{Claims}
\crefname{observation}{Observation}{Observations}
\crefname{conjecture}{Conjecture}{Conjectures}
\crefname{question}{Question}{Questions}
\crefname{remark}{Remark}{Remarks}
\crefname{property}{Property}{Properties}
\crefname{equation}{}{}
\crefrangelabelformat{equation}{\textup{(#3#1#4)--(#5#2#6)}}   %
\crefname{section}{Section}{Sections}
\crefname{figure}{Figure}{Figures}
\crefname{example}{Example}{Examples}
\usepackage{enumitem}
\newlist{myenumerate}{enumerate}{1}
\setlist[myenumerate]{
  label={\itshape(\roman*)},  
  ref={\itshape(\roman*)},    
}
\crefname{myenumeratei}{}{}

\crefname{enumi}{}{}

\usepackage{thmtools,thm-restate}

\title{Circuits in Extended Formulations\thanks{The first author was supported by Air Force Office of Scientific Research grant FA9550-21-1-0233 and NSF grant 2006183, Algorithmic Foundations, Division of Computing and Communication Foundations. The second author was partially supported by the Alexander von Humboldt Foundation with funds from the German Federal Ministry of Education and Research (BMBF).}}
\author{Steffen Borgwardt\inst{1}\orcidID{0000-0002-8069-5046} \and
Matthias Brugger\inst{2}\orcidID{0000-0003-1571-5239}}

\institute{University of Colorado Denver, \email{steffen.borgwardt@ucdenver.edu} \and
Technical University of Munich,
\email{matthias.brugger@tum.de}}

\begin{document}

\maketitle

\begin{abstract}
Circuits and extended formulations are classical concepts in linear programming theory. The circuits of a polyhedron are the elementary difference vectors between feasible points and include all edge directions. We study the connection between the circuits of a polyhedron $P$ and those of an extended formulation of $P$, i.e., a description of a polyhedron $Q$ that linearly projects onto $P$.
It is well known that the edge directions of $P$ are images of edge directions of $Q$. We show that this `inheritance' under taking projections does not extend to the set of circuits. We provide counterexamples with a provably minimal number of facets, vertices, and extreme rays, including relevant polytopes from clustering, and show that the difference in the number of circuits that are inherited and those that are not can be exponentially large in the dimension.
We further prove that counterexamples exist for any fixed linear projection map, unless the map is injective. Finally, we characterize those polyhedra $P$ whose circuits are inherited from all polyhedra $Q$ that linearly project onto $P$. Conversely, we prove that every polyhedron $Q$ satisfying mild assumptions can be projected in such a way that the image polyhedron $P$ has a circuit with no preimage among the circuits of $Q$. Our proofs build on standard constructions such as homogenization and disjunctive programming. 

\end{abstract}

\section{Introduction}

Extended formulations have been widely studied in polyhedral combinatorics and the theory of linear programming, with many recent advances (see, e.g., \cite{kai-11,ccz-13,fmp+-15,fkpt-13,rot-17,kw-15,goe-15,frt-12,af-22} as well as \cite[Chapter 4]{ccz-14} and references therein).
An \emph{extended formulation} of a polyhedron $P$ is a linear system $Ay = b, By \le d$ in variables $y$ defining a polyhedron $Q$ that can be affinely projected onto $P$. The polyhedron $Q$, along with the affine projection that maps $Q$ to $P$, is called an \emph{extension} of $P$. When the projection is clear from the context, for a simple wording, we may refer to $Q$ itself as an extension or an \emph{extension polyhedron}.
The relevance of extended formulations for optimization comes from the fact that one may optimize any linear functional over $P$ by solving a linear program (LP) with feasible region $Q$ instead, which may yield a more compact formulation with fewer constraints. 
Especially for many problems in combinatorial optimization, and LP-based approaches in the form of relaxations, extended formulations have become a powerful tool since the associated polyhedra typically have an exponential number of facets. In some of these cases, significantly smaller extensions with only a polynomial number of facets have been shown to exist \cite{mar-91,yan-91,af-22,won-80,goe-15}. In other situations, such as for the fixed-shape partition polytopes that we discuss below, there does not even exist a known linear formulation \cite{bhr-92,bg-12,bh-17,hor-99}. 

In this paper, we study the connection between the set of circuits of a polyhedron and the set of circuits of an extension. The \emph{circuits} or elementary vectors \cite{r-69} of a polyhedron are the minimal linear dependence relations in its constraint matrix and include all edge directions (we provide a formal definition in \cref{sec:definitions}). Circuit walks, circuit augmentation schemes \cite{dhl-15}, and circuit diameters \cite{bfh-14} generalize the classical concepts of edge walks, the Simplex method, and combinatorial diameters by following steps along the more general set of circuits, so in particular through the interior of a polyhedron. Circuits and circuit walks are of particular interest for polyhedra associated with combinatorial optimization problems, so precisely in the setting where extended formulations are especially useful. This is because the circuits of polyhedra in combinatorial optimization can often be readily interpreted in terms of the underlying application \cite{kps-17}. For example, the circuits for network flow problems are cycles and well-known results like flow decomposition \cite{amo-93,bdfm-18} are immediate consequences of such interpretations.

We are interested in whether the set of circuits of a polyhedron, which may be difficult to describe directly, can be `accessed' through the set of circuits of an extension. We ask the following fundamental question:

\begin{center}\textit{%
When are all circuits of a polyhedron $P$ projections of circuits of a given extension $Q$?}
\end{center}

\noindent As we will see, the connection between the sets of circuits is far weaker than it is for the edge directions. We will be able to quantify this weaker behavior in several ways. Before we explain our main contributions in \cref{sec:outline}, we describe our original motivation for this work in \cref{sec:motivation}, and recall some necessary concepts and introduce convenient notation in \cref{sec:definitions}.

\subsection{Motivation} \label{sec:motivation}

Fixed-shape clustering is the task of partitioning a data set $X$ of $n$ items into $k$ clusters $C_1,\dots,C_k$ such that the number of items in each $C_i$ equals a fixed number $\kappa_i \in \N$, where $\sum_{i=1}^k \kappa_i = n$.
For a data set $X=\{x^{(1)},\dots,x^{(n)}\}\subset \R^d$, popular clustering objectives like least-squares assignments can be found through linear optimization over the so-called {\em fixed-shape partition polytopes} \cite{bhr-92,hor-98,hor-99,for-03,b-10,bg-12,bo-16}.
These are formed as the convex hull of all feasible {\em clustering vectors} $(c^{(1)},\dots,c^{(k)}) \in (\R^d)^k$, where each cluster $C_i$ is represented through a vector $c^{(i)} = \sum_{x \in C_i} x$ of items assigned to it.
For given $X$, $k$ and a vector of cluster sizes $\kappa := (\kappa_1,\dots,\kappa_k)$, we denote these polytopes by $P(X,k,\kappa)$.

Our original motivation for the work in this paper was based on our interest in the circuits of the polytopes $P(X,k,\kappa)$. For example, a characterization of the set of circuits would help with the computation of robustness measures for clusterings \cite{bh-17}, and lead to improved methods for gradual transitions between separable clusterings \cite{bhz-22}.
An explicit inequality description of the fixed-shape partition polytope $P(X,k,\kappa)$ is not known. Instead, computations are performed over a certain transportation polytope, which is given by the following system in variables $y \in \R^{k \times n}$ and which we denote by $T(n,k,\kappa)$: \vspace*{-0.25cm}
\begin{align*}
    \sum_{j=1}^n y_{ij} &= \kappa_i  \qquad\forall i \in [k] \\
    \sum_{i=1}^k y_{ij} &= 1         \;\qquad\forall j \in [n] \\
    y &\ge \zero
\end{align*}
Via the linear map $\pi_X \colon y \mapsto (c^{(1)},\dots,c^{(k)})$ with $c^{(i)} = \sum_{j=1}^n y_{ij} \cdot x^{(j)}$ for all $i \in [k]$, the polytope $T(n,k,\kappa)$ projects to $P(X,k,\kappa)$.

We became interested in whether the set of circuits of $P(X,k,\kappa)$ could be characterized via projecting from $T(n,k,\kappa)$. 
There are a couple of favorable properties of the two polyhedra that made such an approach especially promising \cite{bv-19a,bv-17}; for example, the edges of both polyhedra have near-identical characterizations and all circuit walks in $T(X,k,\kappa)$ are, in fact, edge walks. 
Nonetheless, somewhat surprisingly, in \cref{sec:partpoly} we exhibit that new circuits may appear in the projection onto $P(X,k,\kappa)$.

Our interest in the behavior of circuits and circuit walks under taking projections is further motivated by a well-known fact about \emph{edge} walks: for every edge walk in the original polyhedron, there is an edge walk in the extension that projects onto it. In particular, all edge directions of the original polyhedron are images of edge directions of the extension polyhedron (see \cref{sec:prelim} for a proof). 
In the context of linear programming, there is a pivot rule for the Simplex method that relies precisely on that relationship:
the shadow vertex pivot rule \cite{bor-87}. This pivot rule constructs a Simplex path by following an edge walk in a two-dimensional projection (shadow) of the feasible region of the LP. 
The shadow vertex pivot rule and modifications thereof play an important role in the probabilistic analysis of the Simplex method \cite{dh-20,st-04,ver-09} and in recent work on strong bounds for the performance of the Simplex method on 0/1 polytopes \cite{bdks-21}. Can one exploit similar ideas to design efficient circuit augmentation schemes? Again, this requires an understanding of how circuits and circuit walks behave under taking projections of polyhedra.

\subsection{Notation and Definitions} \label{sec:definitions}

We begin with a formal definition of the set of circuits of a polyhedron, following \cite{r-69,dhl-15,bsy-18,bfh-14,bfh-16}, and then introduce some new terminology for our purposes.

\paragraph{Circuits.}
Let $P = \{ x \in \R^n \colon Ax = b, Bx \le d \}$ be a polyhedron in $\R^n$ where $A \in \R^{p \times n}, B \in \R^{q \times n}$ and $b \in \R^p, d \in \R^q$.
The \emph{circuits} of $P$ with respect to its linear description are the vectors $g \in \ker(A) \setminus \{ \zero \}$ such that $Bg$ is support-minimal (w.r.t.\ inclusion) in $\{ By \colon y \in \ker(A) \setminus \{ \zero \} \}$. Here, $\ker(A)$ denotes the kernel of $A$, $\zero$ denotes the all-zero vector in appropriate dimension, and the \emph{support} of a vector $z \in \R^p$ is the set $\{ i \in [p] \colon z_i \ne 0 \}$. Circuits correspond to directions in the underlying space $\R^n$, and any form of normalization leads to a finite set of unique representatives for the directions. A standard normalization scheme is to assume co-prime integer components (which assumes rational data). However, for our purposes, it will be convenient to not assume any kind of normalization and rather view the set of circuits as a finite union of one-dimensional linear subspaces containing all possible directions.

The set of circuits depends on the description of $P$. We denote it as $\C(A,B)$ when we refer to a specific description; when a system defining $P$ is clear from the context, we use $\C(P)$ in place of $\C(A,B)$. For any description, $\C(P)$ always contains all edge directions of $P$, where an \emph{edge direction} is either the direction of an extreme ray of $P$ or a nonzero multiple of $u-v$ for some pair of adjacent vertices $u,v \in \V(P)$. Here, $\V(P)$ denotes the set of vertices of $P$. 

Unless stated otherwise, we assume a minimal description for $P$ when it is not provided explicitly; this implies that each inequality constraint defines a facet of $P$. This is not a restriction for $A$, as $\ker(A)$ is not affected by the existence of redundant equalities. For $B$, it is standard to assume irredundancy for a different reason: the addition of redundant inequalities to $B$ may lead to a larger set of circuits that are not tied to the geometry of the underlying polyhedron. The benefit of an irredundant description can most easily be seen through an equivalent `geometric' definition of the set of circuits:  $\C(A,B)$ consists of all potential edge directions as the right-hand side vectors $b$ and $d$ vary \cite{g-75}. With an irredundant description of $P$, all facets for any choice of $b$ and $d$ correspond to an original facet of $P$; in turn, $\C(P)$ can be constructed through forming any possible one-dimensional intersection of (translated) facets of $P$ \cite{bv-17}. 

\paragraph{Inheritance under affine projections.}
Let $P \subset \R^n, Q \subset \R^m$ be polyhedra such that $P = \pi(Q)$ for some affine map $\pi \colon \R^m \rightarrow \R^n$. We say that a circuit $g \in \C(P)$ \emph{is inherited from $Q$ under $\pi$} if $g \in \pi(\C(Q)) - \pi(\zero)$,
i.e., if $g$ is the image of a circuit of $Q$ under the `linearized' map $x \mapsto \pi(x) - \pi(\zero)$. If $\pi$ is clear from the context, we use the simpler wording $g \in \C(P)$ is inherited from $Q$.
It is easy to see that $\C(P)$ is unaffected by translations of the polyhedron $P$ (cf.~\cref{prop:bijection} in \cref{sec:prelim}). In our discussion of extensions of $P$, we may therefore restrict ourselves to polyhedra that \emph{linearly} project to $P$ (unless stated otherwise). 

Recall that $P$ and $Q$ are either given through an explicit description or are assumed to have a minimal description.
In the latter case, the assumption of minimality of the implicit linear descriptions is crucial: if we add suitable redundant inequalities to the description of $P$, we may always generate a circuit that is not inherited from $Q$. Likewise, by introducing redundancy to the description of $Q$, one could blow up the set of possible circuits in $Q$ such that all circuits of $P$ would ultimately be inherited from this expanded extended formulation.
Thus, our setting is well-defined to only obtain the strongest statements about the inheritance of circuits.

Most of our results will assume \emph{pointed} polyhedra, i.e., polyhedra given by $Ax = b, Bx \le d$ whose lineality space $\ker(A) \cap \ker(B)$ is trivial. This is appropriate for our purposes since $\C(A,B) = \ker(A) \cap \ker(B) \setminus \{ \zero \}$ if $\ker(A) \cap \ker(B) \ne \{ \zero \}$ and, hence, $\C(\pi(Q)) = \pi(\C(Q))$ is trivially satisfied for any polyhedron $Q$ with a nontrivial lineality space, for all linear maps $\pi$.

\subsection{Contributions and Outline} \label{sec:outline}

Our main contributions are twofold. First, we demonstrate that polyhedra do not necessarily inherit their circuits from extended formulations. This illustrates a fundamental difference between edge directions and circuits. After collecting some tools in \cref{sec:prelim} that simplify our discussion, we first show in \cref{sec:counterexamples} how to construct counterexamples with a minimal number of facets, vertices, and extreme rays. Our construction yields both bounded and unbounded polyhedra in every dimension greater than 2, with corresponding extensions just one dimension higher.

\begin{restatable}{theorem}{thmcounterex} \label{thm:counterex-family}
For all $m,n \in \N$ with $m > n \ge 3$, there exist full-dimensional pointed polyhedra $P \subset \R^n, Q \subset \R^m$ and a linear map $\pi \colon \R^m \rightarrow \R^n$ with $\pi(Q) = P$ such that $\C(P) \not\subseteq \pi(\C(Q))$ and $\C(P) \cap \pi(\C(Q))$ consists precisely of the edge directions of $P$.
Moreover, $Q$ can be chosen to be simple, and $P$ can be chosen to be either a polytope with $n+2$ facets and $n+2$ vertices, or a pointed polyhedral cone with $n+1$ facets and $n+1$ extreme rays.
\end{restatable}

Then, we show that one can even construct counterexamples with an exponential gap in the number of unique directions between the subset of circuits that are inherited and the entire set of circuits.

\begin{restatable}{theorem}{thmexponential} \label{thm:counterex-exponential}
For all $n \ge 3$, there exist pointed polyhedra $P \subset \R^n, Q \subset \R^m$ and a linear map $\pi \colon \R^m \rightarrow \R^n$ with $\pi(Q) = P$ such that $\C(P)$ contains $2^{\Omega(n)}$ unique circuit directions while the number of unique circuit directions in $\C(Q)$ is $O(n^2)$.
\end{restatable}

Our constructions in \cref{sec:counterexamples} range from simple, pathological examples, including zonotopes, to relevant polyhedra from combinatorial optimization. Specifically, we conclude \cref{sec:counterexamples} with a transfer to the fixed-shape partition polytopes from \cref{sec:motivation}, 
where an inheritance fails despite these polyhedra exhibiting a number of favorable properties.

\newcommand{\questionPi}{%
    \textit{Which linear maps $\pi$ have the property that, for \emph{every} polyhedron $Q$, all circuits of $\pi(Q)$ are inherited from $Q$?}
}
\newcommand{\questionP}{%
    \textit{Which polyhedra $P$ inherit their circuits from \emph{every} extension?}
}
\newcommand{\questionQ}{%
    \textit{For which polyhedra $Q$ does \emph{every} polyhedron $P$ that is a linear projection of $Q$ inherit its circuits from $Q$?}
}

As our second contribution, we consider the following three natural questions in \cref{sec:characterization}:
\begin{enumerate}[label=(Q\arabic*),align=left,leftmargin=1.5\parindent,labelwidth=1.5\parindent]
    \item \questionPi \label{qPi}
    \item \questionP  \label{qP}
    \item \questionQ  \label{qQ}
\end{enumerate}

It is not hard to exhibit two sufficient properties for membership in the classes stated in questions \cref{qP,qPi}, respectively. We show that injective maps belong to the class of maps for question \cref{qPi}, and we show that polyhedra in which all circuits are edge directions belong to the class of polyhedra for \cref{qP}. As these observations are useful multiple times for our discussion, the brief arguments can already be found in the preliminaries in \cref{sec:prelim}.

What is more interesting is that these properties are, in fact, both sufficient and {\em necessary}, thus completely resolving questions \cref{qP,qPi}. More precisely, in \cref{sec:characterization}, we obtain the following two results. 
First, we strengthen \cref{thm:counterex-family} and show that, for any linear projection, one can construct a (bounded) counterexample in which no circuits other than the edge directions are inherited, unless the projection map is injective.

\begin{restatable}{theorem}{thmbijection} \label{thm:bijection}
Let $\pi \colon \R^m \rightarrow \R^n$ be a linear map such that $\dim(\pi(\R^m)) \ge 3$. 
Then $\C(\pi(Q)) \subset \pi(\C(Q))$ for all polyhedra $Q \subset \R^m$ if and only if $\pi$ is injective.
In particular, if $\pi$ is not injective, there exists a full-dimensional simple polytope $Q \subset \R^m$ such that $\C(\pi(Q)) \not\subset \pi(\C(Q))$ and $\C(\pi(Q)) \cap \pi(\C(Q))$ consists precisely of the edge directions of $\pi(Q)$.
\end{restatable}

Second, we provide a formal answer to question \cref{qP}.

\begin{restatable}{theorem}{thmedge} \label{thm:edge}
Let $P \subset \R^n$ be a pointed polyhedron. All circuits in $\C(P)$ are edge directions of $P$ if and only if $\C(P) \subset \pi(\C(Q))$ 
for all polyhedra $Q \subset \R^m$ and all linear maps $\pi \colon \R^m \rightarrow \R^n$ with $\pi(Q) = P$.
\end{restatable}

Our proofs of \cref{thm:edge,thm:bijection} in \cref{sec:characterization} are constructive, and build upon observations and tools collected in \cref{sec:prelim,sec:counterexamples}.
Finally, we give a partial answer to question \cref{qQ}, showing that no polyhedron with a non-degenerate vertex has the property that \cref{qQ} asks for.

\begin{restatable}{theorem}{nondegenerate} \label{thm:counterex-nondegenerate}
Let $Q \subset \R^m$ be a polyhedron with $\dim(Q) \ge 4$. If $Q$ has a non-degenerate vertex, then there exists a linear map $\pi \colon \R^m \rightarrow \R^{\dim(Q)-1}$ such that $\pi(Q)$ is full-dimensional and $\C(\pi(Q)) \not\subset \pi(\C(Q))$.
\end{restatable}

In summary, \cref{thm:bijection,thm:edge,thm:counterex-nondegenerate} show that, whenever a polyhedron $P$ inherits all of its circuits from another polyhedron $Q$ (with a non-degenerate vertex) under some affine projection $\pi$, this is not a property of any single one of the three `ingredients' $P$, $Q$, and $\pi$ -- unless inheritance is immediate because $\pi$ defines an affine isomorphism between $P$ and $Q$ or because $P$ has no circuits that are not edge directions. This means that the inheritance of circuits, beyond these simple cases, can only be a property of specific \emph{combinations} of the three ingredients. We conclude \cref{sec:characterization} with a family of examples of such nontrivial combinations that guarantee inheritance, and provide some final remarks in \cref{sec:finalremarks}.

\section{Preliminaries} \label{sec:prelim}

We first note that every polyhedron $P$ has some circuits that are naturally inherited from any extension: the edge directions of $P$. This is a well-known fact about projections of polyhedra. As a service to the reader, we include a brief proof.
\begin{lemma} \label{prop:edge}
Let $P \subset \R^n, Q \subset \R^m$ be polyhedra and let $\pi \colon \R^m \rightarrow \R^n$ be a linear map with $\pi(Q) = P$.
For every edge direction $g$ of $P$, there exists an edge direction $f$ of $Q$ such that $\pi(f)=g$.
\end{lemma}
\begin{proof}
Any face of $P$ is the image of a face of $Q$ under $\pi$; see, e.g., \cite{fkpt-13}. Let $e$ be an edge of $P$ with vertices $v$ and $w$, and let $F$ be a face of $Q$ such that $\pi(F)=e$. Note that $\dim(F) \ge \dim(e) = 1$. Hence, there exist two vertices $v'$ and $w'$ of $F$ with $\pi(v') = v$ and $\pi(w') = w$, and an edge walk connecting $v'$ and $w'$ in $F$. Since $\pi(v'-w') = v-w$, at least one of the edge directions in this walk projects to a nonzero multiple of $v-w$. \qed
\end{proof}

\Cref{prop:edge} shows that, if all circuits of $P$ are edge directions, then $\C(P) \subset \pi(\C(Q))$ for any extension of $P$ specified by $Q$ and $\pi$. This special case includes hypercubes and simplices, and more complicated polytopes such as Birkhoff polytopes and fractional matching polytopes \cite{dks-22,san-18}.

In general, the set of circuits of a polyhedron may of course be much larger than the set of its edge directions. However, we can make a simple a priori observation: two polyhedra that are affinely isomorphic trivially are extensions of one another. Recall that two polyhedra $P \subset \R^p$ and $Q \subset \R^q$ are \emph{affinely (linearly) isomorphic} if there exists an affine (linear) map $\pi \colon \R^q \rightarrow \R^p$ such that $\pi(Q) = P$ and, for all $x \in P$, there exists a unique $y \in Q$ with $\pi(y)=x$. The sets of circuits of affinely isomorphic polyhedra are isomorphic, too, as the next lemma states. 
Since translations are special affine isomorphisms, this justifies our assumption made in \cref{sec:definitions} that all projection maps are \emph{linear} maps.

\begin{lemma} \label{prop:bijection}
Let $Q \subset \R^m$ be a polyhedron and let $\pi \colon \R^m \rightarrow \R^n$ be an affine map.
If $Q$ and $\pi(Q)$ are affinely isomorphic, then $\C(\pi(Q)) = \pi(\C(Q)) - \pi(\zero)$.
\end{lemma}
\begin{proof}
Let $\pi$ be an affine isomorphism between $P := \pi(Q)$ and $Q$. It is easy to see that the affine hulls $\aff(P)$ and $\aff(Q)$ are affinely isomorphic as well, and that there is a one-to-one correspondence between the facets of $P$ and the facets of $Q$. Using the geometric interpretation of the set of circuits (see \cref{sec:definitions}), we obtain the statement. \qed
\end{proof}

Another setting that is easy to resolve is when the polyhedra are low-dimensional. Consider a polyhedron $P$ and an extension $Q$ of $P$ with $\dim(Q) \le 3$. Then either $\dim(P) = 3$ and thus $\dim(Q) = 3$, in which case $P$ and $Q$ must be affinely isomorphic, 
or $\dim(P) \le 2$. In the latter case, every facet of $P$ (if one exists) is an edge of $P$. Hence, every circuit of $P$ trivially is an edge direction.
In summary, we obtain the following corollary to \cref{prop:edge,prop:bijection}.
\begin{corollary}\label{cor:dimension}
Let $P \subset \R^n, Q \subset \R^m$ be polyhedra and let $\pi \colon \R^m \rightarrow \R^n$ be a linear map such that $\pi(Q) = P$.
If $\dim(P) \le 2$ or $\dim(Q) \le 3$, then $\C(P) \subset \pi(\C(Q))$.
\end{corollary}

\Cref{prop:bijection} will be one of the key ingredients for proving \cref{thm:bijection} (see \cref{sec:inheritance-proj}).
Further, the lemma has several interesting implications for certain types of extended formulations and projections: one of the simplifying assumptions commonly made in the study of extended formulations is that the projection is an orthogonal projection onto a subspace of the variables. In our context, such a projection is just as general as any other linear projection.

\begin{corollary} \label{cor:orth-proj}
Let $P \subset \R^n, Q \subset \R^m$ be polyhedra and let $\pi \colon \R^m \rightarrow \R^n$ be a linear map such that $\pi(Q) = P$. 
Further let $Q' := \{ (x,y) \in \R^n \times \R^m \colon x = \pi(y), y \in Q \}$ and $\pi' \colon \R^n \times \R^m \rightarrow \R^n$ defined by $(x,y) \mapsto x$. Then $\pi'(Q') = P$ and $\pi(\C(Q)) = \pi'(\C(Q'))$.
\end{corollary}
\begin{proof}
The claim follows immediately from \cref{prop:bijection}, since the map $\tau \colon y \mapsto (\pi(y),y)$ defines a linear isomorphism between $Q$ and $Q' = \tau(Q)$. \qed
\end{proof}

Another consequence of \cref{prop:bijection} that may be of independent interest is that every pointed polyhedron is affinely isomorphic to a polyhedron in standard form whose set of circuits is isomorphic to the set of circuits of the original polyhedron. 
Recall that a polyhedron $P$ is in \emph{standard form} if it is given by a linear system of the form $Ax = b, x \ge \zero$.

\begin{corollary} \label{cor:slack}
Let $P = \{ x \in \R^n \colon Ax = b, Bx \le d \}$ be a pointed polyhedron where $B \in \R^{m \times n}$. Define the affine map $\sigma \colon \R^n \rightarrow \R^m, x \mapsto d-Bx$. Then $\sigma(P)$ is a polyhedron with a standard form description such that $\C(\sigma(P)) = B \cdot \C(P) = \{ Bg \colon g \in \C(P) \}$. In other words, $\C(\sigma(P))$ is the set of support-minimal vectors in $B \cdot \ker(A)$.
\end{corollary}
\begin{proof}
Let $x,y \in P$ such that $\sigma(x) = \sigma(y)$. Then $A(x-y) = \zero$ and $B(x-y) = \zero$. Since $P$ is pointed, it follows that $x=y$. Hence, $\sigma$ is an isomorphism between $P$ and $\sigma(P)$. Note that $\aff(\sigma(P)) = \sigma(\aff(P))$. We further claim that $\sigma(P) = \aff(\sigma(P)) \cap \R^m_{\ge 0}$ (see also \cite{fkpt-13}). Clearly, $\sigma(P) \subset \aff(\sigma(P)) \cap \R^m_{\ge 0}$. To see that the converse inclusion also holds, let $s \in \aff(\sigma(P)) \cap \R^m_{\ge 0}$, i.e., $s = \sigma(z) \ge \zero$ for some $z \in \aff(P)$. In particular, we have that $Az = b$ and $Bz \le d$, which implies that $z \in P$ as claimed. 
Thus, the description of $\sigma(P)$ as $\aff(\sigma(P)) \cap \R^m_{\ge 0}$ is in standard form. By applying \cref{prop:bijection}, we obtain $\C(\sigma(P)) = \sigma(\C(P)) - d = B \cdot \C(P)$. \qed
\end{proof}

We point out that \cref{cor:slack} contrasts with the behavior of circuits under the standard conversion of a polyhedron $P = \{ x \in \R^n \colon Ax = b, Bx \le d \}$ to standard form: in addition to introducing slack variables $s \ge \zero$ to obtain $Bx+s=d$, one splits each variable $x$ into a positive and a negative part $x=x^+-x^-$, both of which are constrained to be nonnegative. It is shown in \cite{bv-22} that this conversion may introduce exponentially many new circuits. \Cref{cor:slack} suggests that this behavior is a consequence of splitting the variables and not of introducing slack variables, which is what applying the slack map $\sigma$ defined in \cref{cor:slack} implicitly does as well.
More precisely, $\sigma(P)$ is the projection of $P' := \{ (x,s) \in \R^n \times \R^m \colon Ax = b, Bx + s = d, s \ge \zero \}$ onto the slack variables $s$. By \cref{cor:orth-proj}, $P'$ and $P$ (and, hence, $P'$ and $\sigma(P)$) are affinely isomorphic. Characterizing $\sigma(P)$ via $P'$ adds the benefit that one can derive an explicit standard form representation of $\sigma(P)$ from the description of $P'$, using a projection technique found, e.g., in \cite[Theorem 3.46]{ccz-14}:
for a basis $\{ (u^{(1)},v^{(1)}),\dots,(u^{(l)},v^{(l)}) \}$ of $\ker( \begin{pmatrix} B^\top & A^\top \end{pmatrix})$, we have that $\sigma(P) = \{ s \in \R^m \colon s \ge \zero, (u^{(i)})^\top s = (u^{(i)})^\top d \;\forall i \in [l] \}$.

We conclude these preliminaries with a final simple tool that will be useful in the next sections. 
Every circuit of the Cartesian product $P_1 \times P_2$ of polyhedra $P_1$ and $P_2$ is a circuit of one of the product terms, suitably padded with zeros. This was shown in \cite[Lemma 3.9]{bsy-18} for polyhedra in canonical form. We restate and prove the result in all generality here.

\begin{proposition}[\hspace*{-0.15cm}\cite{bsy-18}] \label{lem:cartesian}
Let $P_1 \subset \R^{n_1}, P_2 \subset \R^{n_2}$ be pointed polyhedra. Then
$\C(P_1 \times P_2) = (\C(P_1) \times \{ \zero \}) \cup (\{ \zero \} \times \C(P_2))$.
\end{proposition}
\begin{proof}
Let $P_i = \{ x \in \R^{n_i} \colon A^{(i)} x = b^{(i)}, B^{(i)} x \le d^{(i)} \}$ for $i \in \{1,2\}$. Then $\C(P_1 \times P_2)$ consists precisely of those nonzero vectors $(g^{(1)}, g^{(2)}) \in \ker(A^{(1)}) \times \ker(A^{(2)})$ for which the support of $(B^{(1)} g^{(1)}, B^{(2)} g^{(2)})$ is inclusion-minimal.
Let $(g^{(1)}, g^{(2)}) \in \R^{n_1} \times \R^{n_2}$ be a nonzero vector in $\ker(A^{(1)}) \times \ker(A^{(2)})$. W.l.o.g., we may assume that $g^{(1)} \ne \zero$. Then $(g^{(1)},\zero) \in \ker(A^{(1)}) \times \ker(A^{(2)})$ and the support of $(B^{(1)} g^{(1)}, \zero)$ is contained in the support of $(B^{(1)} g^{(1)}, B^{(2)} g^{(2)})$. Since $P_2$ is pointed, it follows that $(g^{(1)}, g^{(2)}) \in \C(P_1 \times P_2)$ if and only if $g^{(2)} = \zero$ and $g^{(1)} \in \C(P_1)$. \qed
\end{proof}

\section{Counterexamples for the Inheritance of Circuits} \label{sec:counterexamples}

In this section, we prove that, in general, circuits of polyhedra are not inherited from extended formulations.  This contrasts with the behavior for edge directions stated in \cref{prop:edge}. 
We begin by constructing a family of provably minimal counterexamples.
 
\subsection{A Family of Minimal Counterexamples} \label{sec:counterexamples-min}

The essential building block for our constructions is a carefully chosen family of linear projections $\pi_{n,m} \colon \R^m \rightarrow \R^n$ for all $m,n \in \N$ with $m > n \ge 3$. We define a matrix
\[
    \arraycolsep=2.5pt
    \Pi_{n,m} = \left( \begin{array}{cccc|ccc|}
    2 & 1 & 0 & 0 &&&\\
    0 & 0 & 2 & 1 & \zero &&\\
    0 & 1 & 0 & 1 &&&\\
    \hline
    &&& &&&\\
    & \zero && & 2\, \mathbf{I}_{n-3} &&\\
    &&& &&&\\
    \end{array}
    \begin{array}{ccc}
    & \zero &
    \end{array}\right)
    \in \R^{n \times m},
\]
where $\mathbf{I}_d$ denotes the $d \times d$ identity matrix whose rows are the standard unit vectors $\unit{i}$. Let $\pi_{n,m}$ be the map defined by $x \mapsto \Pi_{n,m} x$. This map allows us to state our first family of (unbounded) counterexamples, which all are projections of the nonnegative orthant, a simplicial cone: 

\begin{lemma} \label{lem:counterex-pi-orthant}
Let $m > n \ge 3$ and $\pi := \pi_{n,m}$. Then $\pi(\R^m_{\ge 0})$ is a full-dimensional pointed polyhedral cone with $n+1$ facets and $n+1$ extreme rays. Further, $\C(\pi(\R^m_{\ge 0})) \not\subset \pi(\C(\R^m_{\ge 0}))$ where $\pi(\C(\R^m_{\ge 0}))$ is equal to the set of edge directions of $\pi(\R^m_{\ge 0})$.
\end{lemma}
\begin{proof}
Let $R_n := \pi(\R^m_{\ge 0})$. As a projection of a pointed cone, $R_n$ is a pointed cone with vertex $\zero$ again, spanned by the first $n+1$ column vectors of the matrix $\Pi_{n,m}$. Since $\Pi_{n,m}$ has full row rank, we have that $\dim(R_n) = n$.
We claim that each of these vectors generates an extreme ray of $R_n$, and that $R_n$ is defined by the following $n+1$ inequalities, all of which are facet-defining:
\begin{align*}
  x &\ge \zero \\
  x_1 + x_2 - x_3 &\ge 0
\end{align*}
In order to prove the claim, we proceed by induction on $n$.
The case $n=3$ is easily verified (see \cref{fig:counterex-pi-orthant}). Now let $n \ge 4$. Observe that $\{ x \in R_n \colon x_n = 0 \}$ is a face of $R_n$ which is isomorphic to $R_{n-1}$ and, thus, is a facet. The unique column of $\Pi_{n,m}$ not contained in this facet is the vector $2\unit{n}$, which must therefore generate an extreme ray of $R_n$. All other inequalities except $x_n \ge 0$ define facets of $\{ x \in R_n \colon x_n = 0 \}$ by the induction hypothesis.

Since $n$ of the $n+1$ facets of $R_n$ are defined by nonnegativity constraints, any circuit of $R_n$ cannot be supported in more than two components. This implies that the vectors in $\C(R_n)$ are multiples of $\unit{1}-\unit{2}$, $\unit{3}$, or of one of the (nonzero) column vectors of $\Pi_{n,m}$ (which capture all edge directions of $R_n$). It is easy to see that the set $\pi(\C(\R^m_{\ge 0})$, in turn, consists of multiples of edge directions of $R_n$ only. \qed 
\end{proof}

\tikzset{edge/.style={thick,black},
    hidden edge/.style={dotted,black},
    facet/.style={orange!25,draw=none},
    every node/.style={font=\small}}

\begin{figure}[hbt]%
  \centering
  \begin{tikzpicture}[x=.4cm,y=.4cm]
    \coordinate (e1) at (5,2);
    \coordinate (e2) at (2.5,4.5);
    \coordinate (e3) at (-5,2);
    \coordinate (e4) at (-2.5,4.5);
    
    \fill[draw=none,shade,left color=orange!25,right color=white,shading angle=120] 
        (e1) -- (e2) -- ($1.5*(e2)$) -- ($1.5*(e1)$) -- cycle;
    \fill[draw=none,shade,right color=orange!25,left color=white,shading angle=60]
        (e3) -- (e4) -- ($1.5*(e4)$) -- ($1.5*(e3)$) -- cycle;
    
    \fill[facet] (0,0) -- (e1) -- (e2) -- cycle;
    \fill[facet] (0,0) -- (e3) -- (e4) -- cycle;
    
    \fill[gray!75,draw=none,opacity=.15] (e1) -- (e2) -- (e4) -- (e3) -- cycle;
    
    \draw[edge] (0,0) -- ($1.5*(e1)$);
    \draw[edge] (0,0) -- ($1.5*(e2)$);
    \draw[edge] (0,0) -- ($1.5*(e3)$);
    \draw[edge] (0,0) -- ($1.5*(e4)$);
    \draw[hidden edge,thick] ($1.5*(e1)$) -- ($2*(e1)$);
    \draw[hidden edge,thick] ($1.5*(e2)$) -- ($1.8*(e2)$);
    \draw[hidden edge,thick] ($1.5*(e3)$) -- ($2*(e3)$);
    \draw[hidden edge,thick] ($1.5*(e4)$) -- ($1.8*(e4)$);
    \draw[hidden edge,thick] (e2) -- (e4);
    \draw[hidden edge,thick] (e1) -- (e2);
    \draw[hidden edge,thick] (e3) -- (e4);
    \draw[hidden edge] (e1) -- (e3);
    
    \node[below right] at (e1) {$(2,0,0)$}; 
    \node[right] at (e2) {$(1,0,1)$}; 
    \node[below left] at (e3) {$(0,2,0)$}; 
    \node[left] at (e4) {$(0,1,1)$}; 
    \node[below] at (0,0) {$\zero$};
    
    \foreach \p in {(e1),(e2),(e3),(e4),(0,0)}
        \fill[black] \p circle (.2);
  \end{tikzpicture}
  \caption{The cone $R_3 = \pi_{3,4}(\R^4_{\ge 0})$ from \cref{lem:counterex-pi-orthant}, shown here intersected with the hyperplane $x_1+x_2+x_3 = 2$. The two highlighted facets are defined by $x_1 \ge 0$ and $x_2 \ge 0$, respectively. Their intersection yields the circuit $\unit{3} \in \C(R_3)$. }
  \label{fig:counterex-pi-orthant}
\end{figure}
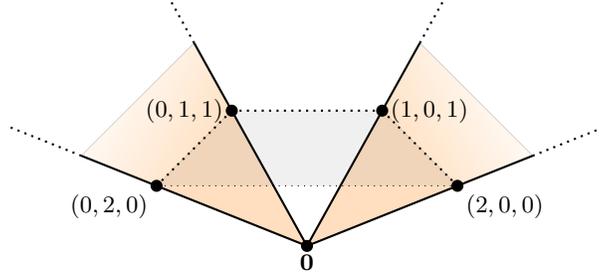

The construction in \cref{lem:counterex-pi-orthant} readily generalizes to the bounded case if we replace the nonnegative orthant $\R^m_{\ge 0}$ with the standard hypercube in $\R^m$, which we denote by $C_m = [0,1]^m$.

\begin{lemma} \label{lem:counterex-pi-zonotope}
Let $m > n \ge 3$ and $\pi := \pi_{n,m}$. Then $\pi(C_m)$ is a full-dimensional polytope and $\C(\pi(C_m)) \not\subset \pi(\C(C_m))$. Moreover, $\pi(\C(C_m))$ consists precisely of the edge directions of $\pi(C_m)$.
\end{lemma}
\begin{proof}
Clearly, $\zero \in \pi(C_m) \subset \R^n_{\ge 0}$. Thus, $\zero$ is a vertex of $\pi(C_m)$. Observe that $\pi(\R^m_{\ge 0})$ is the inner cone of $\pi(C_m)$ at $\zero$. Hence, all $n+1$ facet-defining inequalities of $\pi(\R^m_{\ge 0})$ also define facets of $\pi(C_m)$.
Since both polyhedra are full-dimensional, it follows from the geometric interpretation of the set of circuits that $\C(\pi(C_m)) \supset \C(\pi(\R^m_{\ge 0}))$. Since $\C(C_m) = \C(\R^m_{\ge 0})$, the first part of the statement follows from \cref{lem:counterex-pi-orthant}. 

Note that $\pi(C_m)$ is the linear projection of a hypercube and, as such, a zonotope (see \cref{fig:counterex-pi-zonotope}). Equivalently, $\pi(C_m)$ can be written as the Minkowski sum of $n+1$ line segments $[\zero, \pi_{n,m}(\unit{i})]$ for every $i \in [n+1]$. Since every edge of a zonotope is a translate of one of the line segments from which it is generated, the second part of the statement follows. \qed
\end{proof}

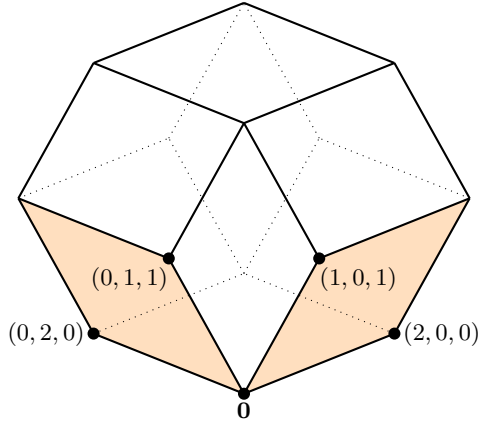
\begin{figure}[hbt]%
  \centering
  \begin{tikzpicture}[x=.4cm,y=.4cm]
    \coordinate (e1) at (5,2);
    \coordinate (e2) at (2.5,4.5);
    \coordinate (e3) at (-5,2);
    \coordinate (e4) at (-2.5,4.5);
    
    \fill[facet] (0,0) -- (e1) -- ($(e1)+(e2)$) -- (e2) -- cycle;
    \fill[facet] (0,0) -- (e3) -- ($(e3)+(e4)$) -- (e4) -- cycle;
    
    \draw[hidden edge] (e1) -- ($(e1)+(e3)$);
    \draw[hidden edge] (e3) -- ($(e1)+(e3)$);
    \draw[hidden edge] ($(e1)+(e3)$) -- ($(e1)+(e2)+(e3)$);
    \draw[hidden edge] ($(e1)+(e3)$) -- ($(e1)+(e3)+(e4)$);
    \draw[hidden edge] ($(e1)+(e2)$) -- ($(e1)+(e2)+(e3)$);
    \draw[hidden edge] ($(e3)+(e4)$) -- ($(e1)+(e3)+(e4)$);
    \draw[hidden edge] ($(e1)+(e2)+(e3)$) -- ($(e1)+(e2)+(e3)+(e4)$);
    \draw[hidden edge] ($(e1)+(e3)+(e4)$) -- ($(e1)+(e2)+(e3)+(e4)$);
    
    \draw[edge] (0,0) -- (e1);
    \draw[edge] (0,0) -- (e2);
    \draw[edge] (0,0) -- (e3);
    \draw[edge] (0,0) -- (e4);
    \draw[edge] (e2) -- ($(e1)+(e2)$);
    \draw[edge] (e2) -- ($(e2)+(e4)$);
    \draw[edge] (e4) -- ($(e3)+(e4)$);
    \draw[edge] (e4) -- ($(e2)+(e4)$);
    \draw[edge] (e1) -- ($(e1)+(e2)$);
    \draw[edge] (e3) -- ($(e3)+(e4)$);
    \draw[edge] ($(e1)+(e2)$) -- ($(e1)+(e2)+(e4)$);
    \draw[edge] ($(e2)+(e4)$) -- ($(e1)+(e2)+(e4)$);
    \draw[edge] ($(e3)+(e4)$) -- ($(e2)+(e3)+(e4)$);
    \draw[edge] ($(e2)+(e4)$) -- ($(e2)+(e3)+(e4)$);
    \draw[edge] ($(e1)+(e2)+(e4)$) -- ($(e1)+(e2)+(e3)+(e4)$);
    \draw[edge] ($(e2)+(e3)+(e4)$) -- ($(e1)+(e2)+(e3)+(e4)$);
    
    \node[right] at (e1) {$(2,0,0)$}; 
    \node[below right] at (e2) {\!\!$(1,0,1)$}; 
    \node[left] at (e3) {$(0,2,0)$}; 
    \node[below left] at (e4) {$(0,1,1)$\!\!}; 
    \node[below] at (0,0) {$\zero$};
    
    \foreach \p in {(e1),(e2),(e3),(e4),(0,0)}
        \fill[black] \p circle (.2);
  \end{tikzpicture}
  \caption{The zonotope $\pi_{3,4}(C_4)$ from \cref{lem:counterex-pi-zonotope}. The two facets that yield the circuit direction $\unit{3}$ are highlighted.}
  \label{fig:counterex-pi-zonotope}
\end{figure}

From the proof of \cref{lem:counterex-pi-zonotope}, we can extract a more general statement about the inheritance of circuits in zonotopes: the only circuits that a zonotope inherits from its hypercube extension are the edge directions.

\begin{corollary}
Let $\pi \colon \R^m \rightarrow \R^n$ be a linear map. Then $\C(\pi(C_m)) \cap \pi(\C(C_m))$ consists precisely of the edge directions of the zonotope $\pi(C_m)$.
\end{corollary}

Let us point out another implication of the previous results. Let $P_1,P_2 \subset \R^n$ be polyhedra. Then the Minkowski sum $P_1 + P_2$ is the image of $P_1 \times P_2$ under the map $\sigma \colon (x,y) \mapsto x+y$. The proof of \cref{lem:counterex-pi-zonotope} implies that, in general, $\C(P_1 + P_2) \not\subseteq \C(P_1) \cup \C(P_2) = \sigma(\C(P_1 \times P_2))$, where the last identity follows from \cref{lem:cartesian}. This means that taking the Minkowski sum of polyhedra may create a circuit that is not a circuit of any of the summands.

Recall that, by \cref{cor:dimension}, $3$ is the minimal dimension of any polyhedron which does not inherit its circuits from an extension. Indeed, the lowest-dimensional counterexamples given in \cref{lem:counterex-pi-orthant,lem:counterex-pi-zonotope} are 3-dimensional.
The family of cones given in \cref{lem:counterex-pi-orthant} is minimal in yet another sense: any $n$-dimensional \emph{unbounded} pointed polyhedron with $n$ facets (and, thus, $n$ extreme rays) is a simplicial cone. Therefore, all circuit directions in such a polyhedron are edge directions, which are naturally inherited from any extension (see \cref{prop:edge}).
In the case of \emph{bounded} polyhedra, any counterexample needs one more facet (and one more vertex), i.e., at least $n+2$ facets and vertices each -- otherwise, it is a simplex and has no circuits that are not also edge directions. 
Even though the zonotopes from \cref{lem:counterex-pi-zonotope} do not satisfy this additional minimality requirement, we can obtain a family of polytopes that are minimal in this sense with a little extra work, using the same projections $\pi_{n,m}$.
To this end, let $S_m$ denote the simplex $S_m = \{x \in \R^m \colon x \ge \zero,\, \sum_{i=1}^m x_i \le 1 \}$ for $m \in \N$.

\begin{lemma} \label{lem:counterex-pi-simplex}
Let $m > n \ge 3$ and $\pi := \pi_{n,m}$. Then $\pi(S_m)$ is a full-dimensional polytope with $n+2$ facets and $n+2$ vertices, and $\C(\pi(S_m)) \not\subset \pi(\C(S_m))$.
Moreover, $\C(\pi(S_m)) \cap \pi(\C(S_m))$ consists precisely of the edge directions of $\pi(S_m)$.
\end{lemma}
\begin{proof}
Let $R_n := \pi(\R^m_{\ge 0})$ and $P_n := \pi(S_m)$.
First observe that $P_n = \{ x \in R_n \colon \sum_{i=1}^n x_i \le 2 \}$ since the entries of each nonzero column of $\Pi_{n,m}$ sum to 2 (see \cref{fig:counterex-pi-orthant}). In particular, this implies that the inequality  $\sum_{i=1}^n x_i \le 2$ is facet-defining for $P_n$, and that all $n+1$ nonzero column vectors of $\Pi_{n,m}$, along with the origin $\zero$, are vertices of $P_n$.
Further note that $\dim(P_n) = \dim(R_n)= n$ and therefore $\C(P_n) \supset \C(R_n)$.

Up to rescaling, $\pi(\C(S_m))$ is the set of all difference vectors between pairs of vertices of $P_n$. 
We claim that every such difference vector that belongs to $\C(P_n)$ is the difference of two adjacent vertices and therefore an edge direction of $P_n$. Indeed, for any $i \ge 4$, the facet $\{ x \in P_n \colon x_i = 0 \}$ contains all vertices of $P_n$ but $2 \unit{i}$. Hence, $2 \unit{i}$ must be adjacent to all other vertices. This implies that the only candidate pairs of non-adjacent vertices are contained in the face of $P_n$ defined by $x_i = 0$ for all $i \ge 4$. Since this face is isomorphic to $P_3$, there are exactly two pairs of non-adjacent vertices, as can be seen in \cref{fig:counterex-pi-orthant}. The corresponding difference vectors are $2\unit{1} - \unit{2} - \unit{3}$ and $2\unit{2} - \unit{1} - \unit{3}$, respectively. Neither of them is a circuit in $\C(P_n)$ because $\unit{1} - \unit{2}$ has strictly smaller support. This implies that every vector in $\C(P_n) \cap \pi(\C(S_m))$ is an edge direction of $P_n$.

It remains to show that $P_n$ has a circuit which is not an edge direction. Since $\C(P_n) \supset \C(R_n)$, the proof of \cref{lem:counterex-pi-orthant} implies that $\unit{3} \in \C(P_n)$. Observe that $P_n$ has no pair of vertices that only differ in the third coordinate and, hence, $\unit{3} \notin \pi(\C(S_m))$. \qed
\end{proof}

Combining \cref{lem:counterex-pi-orthant,lem:counterex-pi-simplex}, we obtain the following theorem.
\thmcounterex*

Even though the counterexamples discussed above may seem pathological, they have an interesting property: for $m=n+1$, the simplicial extensions in \cref{lem:counterex-pi-orthant,lem:counterex-pi-simplex} have the same number of vertices and extreme rays as their projections. Indeed, such a canonical `simplex extension' exists for every pointed polyhedron \cite{fkpt-13} and will be the starting point for proving \cref{thm:edge} in \cref{sec:characterization}.

Next, we will see that not only do there exist counterexamples that fail the inheritance of circuits, but the difference in the number of circuits that are inherited and those that are not may be exponentially large in the dimension.

\subsection{Counterexamples with Many Non-Inherited Circuits}

Our goal is to prove \cref{thm:counterex-exponential}, which we restate for convenience.
\thmexponential*

The proof of \cref{thm:counterex-exponential} relies on the interplay between the basic solutions of a polytope and the circuits in its homogenization. We first recall the relevant concepts.

A \emph{basic solution} of a linear system $Ax = b, Bx \le d$ in variables $x \in \R^n$ is a vector $\overline{x} \in \R^n$ such that the row submatrix of $\begin{pmatrix} A \\ B \end{pmatrix}$ obtained by taking all rows for which $\overline{x}$ satisfies the corresponding constraint at equality has full column rank. Note that $\overline{x}$ need not be feasible.
If $P$ is the polyhedron in $\R^n$ defined by the above system, we will also call $\overline{x}$ a basic solution of $P$. We denote the set of all such basic solutions of $P$ by $\B(P)$. Note that $\B(P)$ contains the set of vertices $\V(P)$ of $P$.
For our purposes, the following equivalent characterization of $\B(P)$ will be convenient.
\begin{lemma} \label{lem:basic-sol}
Let $P = \{ x \in \R^n \colon Ax = b, Bx \le d \}$ be a pointed polyhedron. Then $\B(P)$ is equal to the set of all vectors $g \in \R^n$ such that $Ag = b$ and $Bg - d$ is support-minimal in $\{ By - d \colon y \in \R^n, Ay = b \}$.
\end{lemma}
\begin{proof}
Let $g \in \R^n$ such that $Ag = b$. Note that the matrix $\begin{pmatrix} A \\ B \end{pmatrix}$ has rank $n$ since $P$ is pointed. Therefore, $Bg-d$ is support-minimal in $\{ By - d \colon y \in \R^n, Ay = b \}$ if and only if $\begin{pmatrix} A \\ B' \end{pmatrix}$ has rank $n$, where $B'$ is obtained from $B$ by deleting all rows in the support of $Bg-d$. \qed
\end{proof}

Note the similarity between the above characterization of $\B(P)$ and the definition of the set of circuits $\C(P)$ (cf.~\cref{sec:definitions}). In this sense, one can think of the basic solutions (which include all vertices) as the zero-dimensional analogue of the circuits (which include all edge directions). In fact, we can state this connection more precisely as follows.

Following \cite{zie-95}, we define the \emph{homogenization} of a polyhedron $P = \{ x \in \R^n \colon Ax = b, Bx \le d \}$ as
\begin{equation} \label{eq:hom}\tag{hom}
    \hom(P) := \{ (t,x) \in \R \times \R^n \colon t \ge 0, Ax - bt = \zero, Bx - dt \le \zero \}.
\end{equation}
Observe that $P = \{ x \in \R^n \colon (1,x) \in \hom(P) \}$. If $P$ is pointed, then $\hom(P)$ is a pointed polyhedral cone whose extreme rays are generated by all vectors $(0,g)$, where $g$ is the direction of an extreme ray of $P$, and $(1,v)$ for all vertices $v \in \V(P)$. The circuits of $\hom(P)$ are in correspondence with the basic solutions and circuits of $P$, as shown next.

\begin{lemma} \label{lem:hom-circuits}
Let $P = \{ x \in \R^n \colon Ax = b, Bx \le d \}$ be a pointed polyhedron.
Up to rescaling, the circuits of $\hom(P)$ w.r.t.\ the system \cref{eq:hom} are the nonzero vectors $(\gamma,g) \in \R \times \R^n$ for which one of the following holds:
\begin{myenumerate}
  \item $\gamma = 0$ and $g \in \C(P)$,
    \label{lem:hom-circuits-i}
  \item $\gamma = 1$ and $g \in \B(P)$.
    \label{lem:hom-circuits-ii}
\end{myenumerate}
\end{lemma}
\begin{proof}
Let $(\gamma,g) \in \R \times \R^n$ be a nonzero vector with $Ag - b \gamma = \zero$. If $\gamma = 0$, then $(\gamma,g) \in \C(\hom(P))$ if and only if $Bg$ is support-minimal in $\{ By \colon y \ne \zero, Ay = \zero \}$, i.e., if and only if $g \in \C(P)$.
If $\gamma \ne 0$, we may assume after rescaling that $\gamma = 1$. Suppose that $(1,g) \in \C(\hom(P))$. Then we must have in particular that $Bg-d$ is support-minimal in $\{ By-d \colon y \in \R^n, Ay = b \}$. Hence, $g \in \B(P)$ by \cref{lem:basic-sol}.
Conversely, $(1,g) \in \C(\hom(P))$ if $g \in \B(P)$. To see this, suppose for the sake of contradiction that there exists some vector $y \in \R^n \setminus \{\zero\}$ such that $Ay = \zero$ and the support of $By$ is contained in the support of $Bg-d$. Let $B'$ be the matrix obtained from $B$ by deleting all rows in the support of $Bg-d$. Then $B'y = \zero$. Since $\begin{pmatrix} A \\ B' \end{pmatrix}$ has rank $n$, we must have that $y = \zero$, a contradiction. \qed
\end{proof}

The crucial observation for proving \cref{thm:counterex-exponential} now is the following:
if $P$ is a polytope with vertex set $\V(P)$, then $\hom(P)$ is the image of the nonnegative orthant $\R^{\V(P)}_{\ge 0}$ under the projection $x \mapsto \sum_{v \in \V(P)} x_v (1,v)$. This projection maps the circuits of the nonnegative orthant to the edge directions of $\hom(P)$.
In particular, the number of unique circuit directions of $\R^{\V(P)}_{\ge 0}$ equals $|\V(P)|$ while $\C(\hom(P)) \supset \{1\} \times \B(P)$ by \cref{lem:hom-circuits}. Here it is important to stress that both $\hom(P)$ and $\B(P)$ depend on the particular inequality description of $P$. If we assume a minimal description, then every inequality in \cref{eq:hom} (possibly except $t \ge 0$) defines a facet of $\hom(P)$ and all circuits in \cref{lem:hom-circuits}\cref{lem:hom-circuits-ii} are indeed circuits. To prove \cref{thm:counterex-exponential}, it therefore suffices to exhibit a family of polytopes with polynomially many (in the dimension) vertices but exponentially many basic solutions (w.r.t.\ a minimal description). The corresponding homogenizations will then have an exponential number of unique circuit directions of which only a polynomial number are inherited from the associated nonnegative orthant extension. 
We show that for all $n \ge 2$, the standard cross-polytope $Q_n := \{ x \in \R^n \colon x^\top y \le 1 \;\forall y \in \{-1,1\}^n \}$, suitably cropped by intersecting it with a hypercube, satisfies all the desired properties. This will complete the proof of \cref{thm:counterex-exponential}.
\begin{lemma}
Let $n \ge 2, \delta \in (\frac{1}{2},1)$ and $Q'_n := Q_n \cap [-\delta,\delta]^n$. Then $|\V(Q'_n)| = 4n(n-1)$ and $\B(Q'_n) \supset \{-\delta,\delta\}^n$.
\end{lemma}
\begin{proof}
We first argue that no face of $[-\delta,\delta]^n$ of dimension $n-2$ or less intersects $Q_n$. Indeed, let $F$ be a face of $[-\delta,\delta]^n$ with $\dim(F) \le n-2$. By symmetry, we may assume that $F \subset \{ x \in [-\delta,\delta]^n \colon x_1 = x_2 = \delta \}$. Then the inequality $x_1 + x_2 \le 1$, which is valid for $Q_n$, separates $F$ from $Q_n$ since $\delta > \frac{1}{2}$.

This means that each vertex of $Q'_n$ is contained in at most one facet of $[-\delta,\delta]^n$. In fact, it must be contained in exactly one: none of the vertices of $Q_n$ (which are the positive and negative unit vectors) is contained in $[-\delta,\delta]^n$ since $\delta < 1$. Hence, each vertex of $Q'_n$ is the intersection of exactly one facet of $[-\delta,\delta]^n$ with an edge of $Q_n$.
Observe that every edge of $Q_n$ intersects exactly two distinct facets of $[-\delta,\delta]^n$. Again, this is due to the choice of $\delta < 1$. Hence, the number of vertices of $Q'_n$ equals twice the number of edges of $Q_n$, i.e., $|\V(Q'_n)| = 4n(n-1)$.
Moreover, none of the vertices of $[-\delta,\delta]^n$ is a vertex of $Q'_n$ and for all $i \in [n]$, the inequalities $-\delta \le x_i \le \delta$ are facet-defining for $Q'_n$. Therefore, $\B(Q'_n) \supset \V([-\delta,\delta]^n) = \{-\delta,\delta\}^n$. \qed
\end{proof}

A careful analysis of the counterexamples in \cref{sec:counterexamples-min} shows that the constructions are, in fact, homogenizations, too: consider again the linear map $\pi_{n,m} \colon \R^m \rightarrow \R^n$ from \cref{sec:counterexamples-min}. Define a linear transformation of $\R^n$ which maps $x \in \R^n$ to the vector $x' \in \R^n$ defined by $x'_3 = \frac{1}{2} \sum_{i=1}^n x_i$ and $x'_i = x_i$ for all $i \ne 3$. Under this transformation, the cone $R_n = \pi_{n,m}(\R^m_{\ge 0})$ from \cref{lem:counterex-pi-orthant} can be viewed as the homogenization of some polytope $P \subset \R^{n-1}$ whose vertices are the nonzero column vectors of the matrix $\Pi_{n,m}$ after projecting out the third coordinate. This coordinate takes over the role of the homogeneous coordinate $t$ (recall that all nonzero column vectors of $\Pi_{n,m}$ satisfy $\sum_{i=1}^n x_i = 2$). Then $\zero$ is a basic solution of $P$ which is not a vertex. In the homogenization, $\zero$ yields the circuit $\unit{3}$ by \cref{lem:hom-circuits}.


\subsection{Fixed-Shape Partition Polytopes} \label{sec:partpoly}

All constructions seen so far may seem specifically designed so as to fail the inheritance of circuits. However, there do not only exist pathological counterexamples. We conclude this section with an example from combinatorial optimization that exhibits this undesirable behavior despite a number of favorable properties. Recall the clustering application and the definitions of the associated fixed-shape partition polytopes $P(X,k,\kappa)$ and their extensions $T(n,k,\kappa)$ from \cref{sec:motivation}. 

There are several reasons why one might hope that the set of circuits of $P(X,k,\kappa)$ is inherited from $T(n,k,\kappa)$. 
Most importantly, for any choice of $n,k,\kappa$ and $X$, edges of both $P(X,k,\kappa)$ and $T(n,k,\kappa)$ have a near-identical characterization in terms of the underlying application in the form of `cyclic exchanges' of items between clusters \cite{kw-68,for-03,hor-99,b-10,b-13,bv-19a}: a subset of clusters are ordered in a cycle, and one item from each cluster is transferred to the next along the cycle. Further, $T(n,k,\kappa)$ exhibits a number of interesting properties: it is a $0/1$ polytope whose constraint matrix is totally unimodular and in which all circuits appear as edge directions; in fact, any circuit walk in $T(n,k,\kappa)$ is an edge walk \cite{bv-17}. The special case $k=n$ and $\kappa_i=1$ for all $i\in [n]$ yields the $n$th Birkhoff polytope. Finally, the projection $\pi_X$ is highly symmetric, using the same $x^{(j)}$ in combination with $y_{ij}$ for all $i\in [k]$.

Despite the combination of these many favorable properties, somewhat surprisingly, new circuits may appear in the projection onto $P(X,k,\kappa)$, even for small $d,n$, and $k$.

\begin{lemma}\label{lem:partpoly} 
For all $n \ge 5$, there exist $k \in \N$, $\kappa = (\kappa_1,\dots,\kappa_k) \in \N^k$, and $X \subset \R^{n-2}$ with $|X| = \sum_{i=1}^k \kappa_i = n$ such that $\C(P(X,k,\kappa)) \not\subset \pi_X(\C(T(n,k,\kappa)))$, where $\pi_X \colon \R^{k \times n} \rightarrow (\R^d)^k$ is defined as above.
\end{lemma}
\begin{proof}
For $n \ge 5$, let $X := \V(P_{n-2}) \subset \R^{n-2}$ be the set of vertices of the polytope $P_{n-2} = \pi_{n-2,n-1}(S_{n-1})$ from \cref{lem:counterex-pi-simplex}, where we may assume without loss of generality that $x^{(n)} = \zero$ and the remaining $n-1$ vertices are labelled in arbitrary order. (For given $X$, the set of all possible clustering vectors is invariant under reordering the data points $x^{(j)} \in X$.)

Consider the fixed-shape partition polytope $P(X) := P(X,k,\kappa)$ for $k=2$ and cluster sizes $\kappa_1 = 1$ and $\kappa_2 = n-1$. $P(X)$ is the convex hull of all vectors $(c^{(1)},c^{(2)}) \in \R^{n-2} \times \R^{n-2}$ such that $c^{(1)} = \sum_{j=1}^{n-1} y_{1j} \cdot x^{(j)}$ and $c^{(2)} = \sum_{j=1}^{n-1} x^{(j)} - c^{(1)}$ for some vector $y \in \R^{2 \times n}$ in the corresponding transportation polytope $T := T(n,2,(1,n-1))$, which is described by
\begin{align*}
    \sum_{j=1}^n y_{1j} &= 1 \\
    y_{2j} &= 1-y_{1j} \qquad\forall j \in [n] \\
    y &\ge \zero
\end{align*}

Eliminating the variables $y_{2j}$, it is easy to see that $T$ and the simplex $S_{n-1}$ are affinely isomorphic. Moreover, $P(X)$ is affinely isomorphic to its projection onto the first half $c^{(1)}$ of the clustering vector, which equals $\pi_{n-2,n-1}(S_{n-1}) = P_{n-2}$. The statement then follows immediately from \cref{lem:counterex-pi-simplex,prop:bijection}. \qed
\end{proof}

Despite the negative statement of \cref{lem:partpoly}, we stress that there do exist classes of fixed-shape partition polytopes in which all circuits are inherited from the transportation-style extensions, even though this happens for one of the two trivial reasons stated in \cref{prop:edge,prop:bijection}.
For instance, $T(n,k,\kappa)$ is a fixed-shape partition polytope itself for any $n$ and $k$, using the standard unit vectors in $\R^n$ as item locations, as already observed in \cite{bv-17}. Similarly, suppose that we augment a given data set $X \in \R^d$ of size $|X|=n$ with the unit vectors in $\R^n$, i.e., we replace each item location $x^{(i)}$ with $(x^{(i)},\unit{i}) \in \R^d \times \R^n$ for all $i \in [n]$. Then the fixed-shape partition polytope resulting from this augmented embedding can equivalently be derived using the construction in \cref{cor:orth-proj}. In particular, the resulting polytope is affinely isomorphic to $T(n,k,\kappa)$.

\section{The Role of Projection Maps and Polyhedra for the Inheritance of Circuits} \label{sec:characterization}

In the previous section, we saw that there exist polyhedra that do not inherit all their circuit directions from an extension.
In this section, we explore the role that the individual `ingredients' of those counterexamples -- the original polyhedron $P$, the extension polyhedron $Q$, and the projection map $\pi$ from $Q$ to $P$ -- play for the inheritance of circuits.
Our discussion is driven by three natural questions, first stated in \cref{sec:outline}:
\begin{enumerate}[label=(Q\arabic*),align=left,leftmargin=1.5\parindent,labelwidth=1.5\parindent]
    \item \questionPi 
    \item \questionP  
    \item \questionQ  
\end{enumerate}

We provide a complete characterization of the maps for question \cref{qPi} in \cref{sec:inheritance-proj} and of the polyhedra for question \cref{qP} in \cref{sec:inheritance-poly}. As we will see, they correspond to restrictive properties that make the inheritance of circuits trivial. We further provide a partial answer to question \cref{qQ} in \cref{sec:inheritance-ext}. In \cref{sec:bestpossible}, we explain why our characterizations are best possible. We do so through the discussion of some combinations of polyhedra and maps that lead to an inheritance of circuits, but where neither of them exhibits the aforementioned properties.

\subsection{Inheritance Based on the Projection Map} \label{sec:inheritance-proj}

In \cref{prop:bijection}, we saw that linear isomorphisms essentially preserve the set of circuits. We first show that no other type of linear map guarantees inheritance for \emph{all} polyhedra, thus resolving \cref{qPi}.
\thmbijection*

Recall from \cref{thm:counterex-family} that in every dimension greater than 2, there are polyhedra that do not inherit their circuits from all extensions.
The key observation for proving \cref{thm:bijection} will be that, in any fixed dimension, the particular projection used to obtain \cref{thm:counterex-family} can be exchanged for any other one after a suitable linear transformation of the domain space. 

\begin{proof}[of \cref{thm:bijection}]
By \cref{prop:bijection}, it suffices to show the `if' part of the statement and we may assume that $\pi(\R^m) = \R^n$.
If $\pi$ is not injective, then $m > n$. By \cref{thm:counterex-family}, there exists a full-dimensional simple polytope $Q \subset \R^m$ and a linear map $\sigma \colon \R^m \rightarrow \R^n$ such that the polytope $P := \sigma(Q) \subset \R^n$ is full-dimensional, $\C(P) \not\subseteq \sigma(\C(Q))$, and the set of edge directions of $P$ is precisely the set $\C(P) \cap \sigma(\C(Q))$. 
Since $\dim(P) = n$, we have that $\sigma(\R^m) = \R^n = \pi(\R^m)$. Hence, there exists a linear transformation $\tau \colon \R^m \rightarrow \R^m$ such that $\pi = \sigma \circ \tau$.
Now consider the polytope $\widetilde{Q} := \tau^{-1}(Q)$. Clearly, $\widetilde{Q}$ is simple again with $\dim(\widetilde{Q}) = m$ and
\[
  \pi(\widetilde{Q}) = (\sigma \circ \tau \circ \tau^{-1}) (Q) = \sigma(Q) = P.
\]
Using \cref{prop:bijection}, we conclude that
\[
    \pi(\C(\widetilde{Q})) 
    = (\sigma \circ \tau) (\C(\widetilde{Q}))
    = \sigma(\C(Q)) \not\supset \C(P)
\]
and, thus,  $\C(P) \cap \pi(\C(\widetilde{Q})) = \C(P) \cap \sigma(\C(Q))$. \qed
\end{proof}

\subsection{Inheritance for all Extensions} \label{sec:inheritance-poly}

Next, we resolve \cref{qP} by showing that any polyhedron which inherits its circuits from {\em every} extension cannot have a circuit that is not an edge direction already.
\thmedge*

In fact, we will prove a slightly stronger statement that clearly implies \cref{thm:edge}:

\begin{theorem} \label{thm:single-edge}
Let $P \subset \R^n$ be a pointed polyhedron and $g \in \R^n \setminus \{\zero\}$. $g$ is an edge direction of $P$ if and only if $g \in \pi(\C(Q))$
for all polyhedra $Q \subset \R^m$ and all linear maps $\pi \colon \R^m \rightarrow \R^n$ with $\pi(Q) = P$.
\end{theorem}

For any pointed polyhedron $P$ and a nonzero vector $g$ which is not among the edge directions of $P$, we construct an extension of $P$ none of whose circuits projects to $g$. The construction is based on a classical result of Balas \cite{bal-85,bal-98} on the union of polyhedra, which we recall next. We denote the union over a family $\mathcal{P}$ of sets as $\bigcup \mathcal{P}$.

\begin{proposition}[\hspace*{-0.15cm}\cite{bal-85,bal-98}] \label{prop:balas}
Let $P \subset \R^n$ be a polyhedron, and let $\mathcal{P} = \{ P_1,\dots,P_p \}$ be a family of nonempty polyhedra 
$P_i = \{ x \in \R^n \colon A^{(i)} x = b^{(i)}, B^{(i)} x \le d^{(i)} \}, i \in [p]$,
such that $P = \conv \left( \bigcup \mathcal{P} \right)$.
Consider the polyhedron $Q_{\mathcal{P}} \subset \R^p \times (\R^n)^p$ defined by the following linear system in variables $\lambda \in \R^p$ and $x^{(i)} \in \R^n$ for all $i \in [p]$:
\begin{align*}
  \lambda &\ge \zero  \\
  \sum_{i=1}^p \lambda_i &= 1  \tag{disj}\label{eq:balas4}\\
  A^{(i)} x^{(i)} &= b^{(i)} \lambda_i   \qquad\forall i \in [p] \\
  B^{(i)} x^{(i)} &\le d^{(i)} \lambda_i \qquad\forall i \in [p] 
\end{align*}
Then $P = \left\{ \sum_{i=1}^p x^{(i)} \colon (\lambda,x^{(1)},\dots,x^{(p)}) \in Q_{\mathcal{P}} \right\}$.
\end{proposition}

Next, we give a characterization of the circuits of the extension $Q_{\mathcal{P}}$ defined in \cref{prop:balas}.

\begin{lemma} \label{lem:balas-circuits}
Let $P \subset \R^n$ be a pointed polyhedron and let $\mathcal{P}$ and $Q_{\mathcal{P}}$ be defined as in \cref{prop:balas}.
Up to rescaling, the circuits of $Q_{\mathcal{P}}$ w.r.t.\ the system \cref{eq:balas4} are the nonzero vectors $(f,g^{(1)},\dots,g^{(p)}) \in \R^p \times (\R^n)^p$ for which one of the following holds:
\begin{myenumerate}
  \item $f = \zero$;
    $g^{(i)} \in \C(P_i)$ for some $i \in [p]$
    and $g^{(k)} = \zero$ for all $k \ne i$,
    \label{lem:balas-circuits-i}
  \item $f = \mathbf{e}_i - \mathbf{e}_j$ for some $i,j \in [p], i \ne j$;
    $g^{(i)} \in \B(P_i), g^{(j)} \in \B(P_j)$,
    and $g^{(k)} = \zero$ for all $k \ne i,j$.
    \label{lem:balas-circuits-ii}
\end{myenumerate}
\end{lemma}
\begin{proof}
Let $(f,g^{(1)},\dots,g^{(p)}) \in \R^p \times (\R^n)^p$ be a circuit of $Q_{\mathcal{P}}$. If $f = \zero$, then $(g^{(1)},\dots,g^{(p)}) \in \C(P_1 \times \dots \times P_p)$. Statement \cref{lem:balas-circuits-i} immediately follows from an inductive application of \cref{lem:cartesian}.

Now suppose that $f \ne \zero$. Since $\sum_{i=1}^p f_i = 0$, $f$ must be supported in at least two components, say, $f_1 \ne 0$ and $f_2 \ne 0$. We claim that these are the only nonzero components of $f$, and that $g^{(3)} = \dots = g^{(p)} = \zero$. 
Then, after rescaling, we may assume that $f_1 = -f_2 = 1$, and statement \cref{lem:balas-circuits-ii} follows from \cref{lem:basic-sol}.
In order to prove the claim, observe that the vector $(\unit{1} - \unit{2}, \frac{1}{f_1} g^{(1)}, -\frac{1}{f_2} g^{(2)}, \zero, \dots, \zero) \in \R^p \times (\R^n)^p$ belongs to $\C(Q_{\mathcal{P}})$, too. Then $f$ and $\unit{1}-\unit{2}$ must have the same support. Further, by support-minimality of $B^{(k)} g^{(k)} - d^{(k)} f_k$, we cannot have that $g^{(k)} \ne \zero$ for some $k \ge 3$ as all $P_k$ are pointed. \qed
\end{proof}

We are now ready to prove the main result of this section.

\begin{proof}[of \cref{thm:single-edge}]
The `only if' part immediately follows from \cref{prop:edge}. For the converse implication, suppose that $g$ is not an edge direction of $P$. We first show the statement for the case that $P$ is a polytope.

Let $\mathcal{U} := \{ \{u,v\} \colon u,v \in \V(P),\, g \in \R (u-v) \}$ be the set of unordered pairs $\{u,v\}$ of vertices of $P$ whose difference is in direction of $g$ (possibly $\mathcal{U} = \emptyset$).
Observe that the pairs in $\mathcal{U}$ are pairwise disjoint: if $\{u,v\},\{v,w\} \in \mathcal{U}$ then $u,v,w$ are collinear. Since all three of them are vertices, it follows that $u=w$.
For every pair $\{u,v\} \in \mathcal{U}$, let $F_{\{u,v\}}$ be the minimal face of $P$ containing both $u$ and $v$. Since $u-v$ is not an edge direction of $P$, we have $\dim(F_{\{u,v\}}) \ge 2$. Hence, there exists a vector $z \in \R^n \setminus \{\zero\}$ orthogonal to $u-v$ such that $\frac{u+v}{2} \pm \varepsilon z \in F_{\{u,v\}}$ for some small $\varepsilon > 0$. This means that the parallelogram 
\[
  P_{\{u,v\}} := \conv \left\{ u,v,\; \frac{u+v}{2} + \varepsilon z,\; \frac{u+v}{2} - \varepsilon z \right\}
\]
is contained in $F_{\{u,v\}}$.

Moreover, for all $u,v \in \V(P)$ with $u \ne v$ and $\{u,v\} \notin \mathcal{U}$, there exist $a \in \R^n \setminus \{\zero\}, \beta \in \R$ such that $a^\top g = 0$ and $a^\top u < \beta < a^\top v$, i.e., the hyperplane $\{ x \in \R^n \colon a^\top x = \beta \}$ strictly separates $u$ and $v$. If $\varepsilon$ is sufficiently small, then also $P_{\{u,v\}}$ and $P_{\{u',v'\}}$, and $P_{\{u,v\}}$ and $w$ can be strictly separated by a hyperplane whose normal vector is orthogonal to $g$, for any distinct $\{u,v\},\{u',v'\} \in \mathcal{U}$ and $w \in \V(P) \setminus \bigcup \mathcal{U}$.

Now define
\[
  \mathcal{P} := 
    \left\{ P_{\{u,v\}} \colon \{u,v\} \in \mathcal{U} \right\}
    \cup 
    \left\{ \{v\} \colon v \in \V(P) \setminus \bigcup \mathcal{U} \right\}.
\]
Since $\mathcal{P}$ is a family of polytopes (singletons and parallelograms) contained in $P$ which covers $\V(P)$, we have that $P = \conv \left( \bigcup \mathcal{P} \right)$. For this family $\mathcal{P}$, we consider the extension $Q_{\mathcal{P}}$ as defined in \cref{prop:balas}. We claim that under the projection given in \cref{prop:balas}, none of the circuits of $Q_{\mathcal{P}}$ maps to a multiple of $g$.
By \cref{lem:balas-circuits}, the circuits of $Q_{\mathcal{P}}$ either map to \cref{lem:balas-circuits-i} edge directions of some member of the family $\mathcal{P}$ or to \cref{lem:balas-circuits-ii} (scaled) differences of two vertices of different members of $\mathcal{P}$. This is because $\dim(Q) \le 2$ for all $Q \in \mathcal{P}$ and all basic solutions of parallelograms are vertices.
By construction, none of the parallelograms $P_{\{u,v\}}$ has an edge in direction $g$, ruling out case \cref{lem:balas-circuits-i}. For case \cref{lem:balas-circuits-ii}, recall that any two distinct members of $\mathcal{P}$ can be strictly separated by some hyperplane whose normal vector is orthogonal to $g$. We conclude that $g$ is not inherited from $Q_{\mathcal{P}}$.

Now suppose that $P$ is unbounded. Then $P = P' + \rec(P)$ where $P' := \conv(\V(P))$ and $\rec(P)$ denotes the recession cone of $P$. We first define an extension for each Minkowski summand individually and then combine the two. 
Indeed, since the first summand $P'$ is a polytope, the first part of the proof yields an extension $Q_{\mathcal{P}'}$ of $P'$ none of whose circuits is sent to $g$. 
Suppose that the second summand $\rec(P)$ is generated by $q$ extreme rays in directions $\{ r^{(1)},\dots,r^{(q)} \} \subset \R^n$.
Then it is the image of the nonnegative orthant $\R^q_{\ge 0}$ under the map $\R^q \ni y \mapsto \sum_{i=1}^q y_i r^{(i)} \in \R^n$. By assumption, $g$ is not a multiple of $r^{(i)}$ for any $i \in [q]$. Now consider the polyhedron $Q_{\mathcal{P}'} \times \R^q_{\ge 0}$. It is an extension of $P$, where the corresponding projection first maps $Q_{\mathcal{P}'} \times \R^q_{\ge 0}$ to $P' \times \rec(P)$ and then applies the map $(x,y) \mapsto x+y$. By \cref{lem:cartesian}, $g$ is not inherited from $Q_{\mathcal{P}'} \times \R^q_{\ge 0}$ under this combined map. This concludes the proof. \qed
\end{proof}

We stress that the extension $Q_{\mathcal{P}}$ constructed in the above proof is not necessarily given by an irredundant system if we follow \cref{prop:balas}. However, for the purpose of proving a negative result about the \emph{non}-inheritance of a particular direction, this is not a restriction.

Before we focus on the last of the three ingredients, the extension polyhedron $Q$, let us remark that the simplex extension that we saw in \cref{lem:counterex-pi-simplex} is, in fact, a special case of the more general extension $Q_{\mathcal{P}}$ used in the proof of \cref{thm:edge}: for a polytope $P$ and the decomposition $\mathcal{P} := \{ \{v\} \colon v \in \V(P) \}$, the polyhedron $Q_{\mathcal{P}}$ is affinely isomorphic to the simplex $S_{|\V(P)|-1}$.
\Cref{prop:balas} also generalizes another result from \cref{sec:counterexamples}: let $P$ be a polytope and consider $\mathcal{P} := \{ \{\zero\}, \{1\} \times P \}$; then $Q_{\mathcal{P}}$ as defined in \cref{prop:balas} equals $\{ (t,x) \in \hom(P) \colon t \le 1 \}$.

\subsection{(No) Inheritance Based on the Extension Polyhedron} \label{sec:inheritance-ext}

In \cref{sec:counterexamples-min}, we saw that there exist polyhedra -- simplices, simplicial cones, and hypercubes in dimension 4 and greater -- that can be projected in such a way that not all circuits of the image polyhedron are inherited from the original one. In this section, we prove that these polyhedra can essentially be exchanged for any other polyhedron $Q$ (of the same dimension), provided that $Q$ has a non-degenerate vertex.
\nondegenerate*

Before we give a detailed proof of this result, let us take a closer look at the proofs of \cref{lem:counterex-pi-orthant,lem:counterex-pi-zonotope,lem:counterex-pi-simplex} and identify a common theme: all polyhedra that we projected from in \cref{sec:counterexamples-min} have a non-degenerate vertex at the origin $\zero$, and their inner cone at $\zero$ equals the nonnegative orthant. So in any fixed dimension, they are all identical \emph{locally} at $\zero$. We then applied a carefully chosen linear projection map which preserves this local resemblance. This allowed us to always generate a particular circuit direction $\unit{3}$, for which we were then able to establish non-inheritance. In this last step, however, knowledge of the set of circuits of the extension polyhedron was crucial. This will be the major technical challenge when applying the above proof strategy to an arbitrary polyhedron $Q$: neither do we know the other facets of $Q$ that are not incident with $\zero$ nor is $\C(Q)$ given explicitly.
We address this challenge by defining an infinite family of linear projections such that every member of the family maps $Q$ to a polyhedron with vertex $\zero$ in which the non-inherited circuit direction $\unit{3}$ from the results in \cref{sec:counterexamples-min} still appears as a circuit. Moreover, the family will have the property that no nonzero vector is sent to $\unit{3}$ (or a multiple thereof) under more than one of the projections in the family. Since our family is \emph{infinite} but $\C(Q)$ is \emph{finitely} generated, there must be some member of the family which does not send any of the circuits of $Q$ to $\unit{3}$. This will be the map that we can apply to $Q$ and obtain the same non-inheritance result as in \cref{sec:counterexamples-min}. 
The remainder of this section is dedicated to the proof details.

\begin{proof}[of \cref{thm:counterex-nondegenerate}]
After an affine transformation, we may assume that $Q$ is full-dimensional, $\zero$ is a non-degenerate vertex of $Q$, and the inner cone of $Q$ at $\zero$ equals $\R^m_{\ge 0}$. 
For all $\alpha \in \N \setminus \{1\}$, we define the matrix
\[
    \arraycolsep=3.5pt
    \Pi_\alpha := \left( \begin{array}{cccc|ccc}
    \alpha & 1 & 0 & 0 &&&\\
    0 & 0 & \alpha & 1 & \zero &&\\
    0 & \alpha-1 & 0 & \alpha-1 &&&\\
    \hline
    &&& &&&\\
    & \zero && & \alpha\, \mathbf{I}_{n-3} &&\\
    &&& &&&\\
    \end{array}
    %
    \right)
    \in \R^{(m-1) \times m}
\]
and a corresponding linear map $\pi_\alpha \colon \R^m \rightarrow \R^{m-1}, x \mapsto \Pi_\alpha x$. Note that $\pi_2 = \pi_{m-1,m}$ where $\pi_{m-1,m}$ is the projection used in \cref{sec:counterexamples}.

Consider the cone $\pi_\alpha(\R^m_{\ge 0}) \subset \R^{m-1}$. It is defined by the $m$ inequalities
\begin{align*}
  x &\ge \zero \\
  (\alpha-1) x_1 + (\alpha-1) x_2 - x_3 &\ge 0
\end{align*}
all of which are facet-defining. This can be seen using the same arguments as in the proof of \cref{lem:counterex-pi-orthant}. In fact, the cone above is obtained from $\pi_2(\R^m_{\ge 0})$ by rescaling it along the third coordinate. In particular, $\unit{3} \in \C(\pi_\alpha(\R^m_{\ge 0}))$. Now consider $\pi_\alpha^{-1}(\R \unit{3}) =: K_\alpha$, i.e., $K_\alpha$ is the set of all vectors in $\R^m$ that $\pi_\alpha$ sends to a multiple of $\unit{3} \in \R^{m-1}$. 
Since $\Pi_\alpha$ has full row rank, $K_\alpha$ is a two-dimensional linear subspace of $\R^m$, spanned by the vectors $\alpha \unit{2} - \unit{1}$ and $\alpha \unit{4} - \unit{3}$.
For any $\alpha \ne \beta$, we have that $K_\alpha \cap K_\beta = \{\zero\}$ because the four basis vectors of $K_\alpha$ and $K_\beta$ are linearly independent. Now recall from the definition of the set of circuits that $\C(Q)$ consists of a \emph{finite} number of one-dimensional linear subspaces of $\R^m$. Hence, there must exist some $\alpha \ne 1$ such that $\C(Q) \cap K_\alpha = \emptyset$. For this choice of $\alpha$, we conclude that $\unit{3} \notin \pi_\alpha(\C(Q))$ while $\unit{3} \in \C(\pi_\alpha(\R^m_{\ge 0})) \subset \C(\pi_\alpha(Q))$, where the last inclusion follows from the fact that $\pi_\alpha(\R^m_{\ge 0})$ is the inner cone of $\pi_\alpha(Q)$ at the vertex $\zero$. \qed
\end{proof}

\subsection{Inheritance in Nontrivial Instances} \label{sec:bestpossible}

\Cref{thm:bijection,thm:edge,thm:counterex-nondegenerate} imply that, beyond the trivial cases that we saw in \cref{sec:prelim}, inheritance of circuits cannot be a property of a polyhedron, of a specific extension polyhedron, or of the map between the two by itself. We conclude our discussion by showing that there do exist instances for which a combination of these three ingredients leads to the desired inheritance of circuits while each individual ingredient does not satisfy the restrictive assumptions of \cref{thm:bijection,thm:edge}. In this sense, \cref{thm:bijection,thm:edge} are the best possible statements.

\begin{lemma} \label{lem:counterex-pi-simplex-modified}
For all $m,n \in \N$ with $n \ge 3$ and $m \ge n+3$, there exist full-dimensional polytopes $P \subset \R^n, Q \subset \R^m$ and a linear map $\pi \colon \R^m \rightarrow \R^n$ with $\pi(Q) = P$ such that $\C(P) \subset \pi(\C(Q))$, $P$ and $Q$ are not linearly isomorphic, and not all circuits of $P$ are edge directions.
\end{lemma}
\begin{proof}
We again modify the construction from \cref{lem:counterex-pi-simplex}. For $n \ge 3$ and $m \ge n+3$, we define the matrix
\[
    \arraycolsep=2.5pt
    \Pi'_{n,m} := \left( \begin{array}{cccccc|ccc|}
    1 & 1 & 2 & 0 & 0 & 0 &&&\\
    0 & 0 & 0 & 1 & 1 & 2 & \zero &&\\
    0 & 1 & 1 & 0 & 1 & 1 &&&\\
    \hline
    &&&&& &&&\\
    && \zero &&& & \mathbf{I}_{n-3} &&\\
    &&&&& &&&\\
    \end{array}
    \begin{array}{ccc}
    & \zero &
    \end{array}\right)
    \in \R^{n \times m}
\]
and let $\pi' \colon \R^m \rightarrow \R^n$ be the linear map $x \mapsto \Pi'_{n,m} x$.
Consider the polytope $P'_n := \pi'(S_m)$ (see \cref{fig:counterex-pi-simplex-modified}). We claim that $P'_n$ is defined by the following irredundant system of inequalities:
\begin{align*}
  x &\ge \zero \\
  x_3 &\le 1 \\
  x_1 + x_2 - x_3 &\ge 0 \\
  x_1 + x_2 - x_3 &\le 1
\end{align*}
All inequalities are valid for $P'_n$. To see that they are facet-defining, we proceed by induction on $n$, similar to the proof of \cref{lem:counterex-pi-simplex}. For the case $n = 3$, we refer to \cref{fig:counterex-pi-simplex-modified}. If $n \ge 4$, then $\{ x \in P'_n \colon x_n = 0 \}$ is a face of $P'_n$ which contains all column vectors of $\Pi'_{n,m}$ but $\unit{n}$. Hence, it is isomorphic to $P'_{n-1}$ and therefore is a facet of $P'_n$. 

In particular, since $P'_n$ has more than $n+1$ facets, it is not a simplex and, thus, cannot be isomorphic to $S_m$.
Consider again the cone $R_n \subset \R^n$ defined in the proof of \cref{lem:counterex-pi-orthant}. It is easy to see that $\C(P'_n) = \C(R_n)$. Hence, after rescaling, every circuit direction of $P'_n$ appears as one of the column vectors of $\Pi'_{n,m}$ or as the difference of two of them. This implies that $\C(P'_n) \subset \pi'(\C(S_m))$.
Further, $\unit{3} \in \C(P'_n)$ is not an edge direction of $P'_n$, and $S_m$ clearly has a non-degenerate vertex. \qed 
\end{proof}

\begin{figure}[hbt]%
  \centering
  \begin{tikzpicture}[x=.4cm,y=.4cm] 
    \coordinate (e1) at (2.5,2);
    \coordinate (e2) at (2.5,4.5);
    \coordinate (e3) at (-2.5,2);
    \coordinate (e4) at (-2.5,4.5);
    
    
    \draw[hidden edge] (e1) -- (e3);
    
    \draw[edge] (0,0) -- (e1);
    \draw[edge] (0,0) -- (e2);
    \draw[edge] (0,0) -- (e3);
    \draw[edge] (0,0) -- (e4);
    \draw[edge] (e1) -- ($(e1)+(e2)$);
    \draw[edge] (e2) -- ($(e1)+(e2)$);
    \draw[edge] (e3) -- ($(e3)+(e4)$);
    \draw[edge] (e4) -- ($(e3)+(e4)$);
    \draw[edge] (e2) -- (e4);
    \draw[edge] ($(e1)+(e2)$) -- ($(e3)+(e4)$);
    
    \node[right] at (e1) {$(1,0,0)$}; 
    \node[above left] at (e2) {$(1,0,1)$\!\!\!\!\!}; 
    \node[left] at (e3) {$(0,1,0)$}; 
    \node[above right] at (e4) {\!\!\!\!\!$(0,1,1)$}; 
    \node[below] at (0,0) {$\zero$};
    \node[right] at ($(e1)+(e2)$) {$(2,0,1)$}; 
    \node[left] at ($(e3)+(e4)$) {$(0,2,1)$}; 
    
    \foreach \p in {(e1), (e2), (e3), (e4), (0,0), ($(e1)+(e2)$), ($(e3)+(e4)$)}
        \fill[black] \p circle (.2);
  \end{tikzpicture}
  \caption{The polytope $P'_3$ from the proof of \cref{lem:counterex-pi-simplex-modified}.}
  \label{fig:counterex-pi-simplex-modified}%
\end{figure}
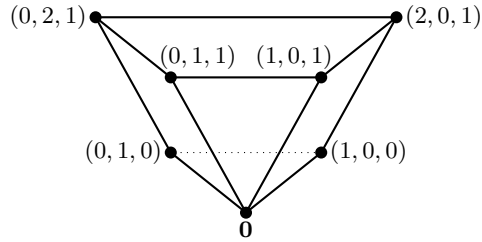

\section{Final Remarks} \label{sec:finalremarks}

We showed that the connection between the sets of circuits of polyhedra and their extensions is much weaker than the connection between their edge directions: in general, circuits are not inherited under affine projections. 
Whenever this does happen for a nontrivial combination of two polyhedra and a projection map between them, it is due to the specific combination of the three ingredients and not due to any single one of them by itself. Therefore, a natural direction of future work would be to identify properties of these combinations that are sufficient for the inheritance of circuits (beyond our characterizations in \cref{sec:characterization}). 

\bibliographystyle{plain}


\end{document}